\documentclass[11pt,a4paper]{article}

\usepackage{theorem}%
\usepackage{amsmath,amssymb,amsfonts}%
\usepackage{graphicx}      
\usepackage{color,dsfont}
\usepackage{enumerate,upgreek}
\definecolor{mycol}{rgb}{0,0,1}
\definecolor{mcc}{rgb}{0,0.4,0.6}
\usepackage{amsmath,amssymb}
\usepackage[mathscr]{eucal}
\usepackage{times,pifont}

\allowdisplaybreaks%
\theoremstyle{change}%

\newtheorem{definition}{Definition}[section]%
\newtheorem{theorem}[definition]{Theorem}%
\newtheorem{proposition}[definition]{Proposition}%
\newtheorem{lemma}[definition]{Lemma}%
\newtheorem{corollary}[definition]{Corollary}%
\newtheorem{ass}[definition]{Assumption}%
{\theorembodyfont{\rmfamily} \newtheorem{remark}[definition]{Remark}}%
{\theorembodyfont{\rmfamily} }%

\newenvironment{proof}
  {{\bf Proof\hspace{0.1cm}}}
  {\qquad \hspace*{\fill} $\Box$}%

\newcommand{\ep}{\varepsilon}
\newcommand{\tm}{\times}
\newcommand{\tr}{\mathrm{tr}}
\newcommand{\Gl}{\mathrm{Gl}}
\newcommand{\fa}{\mathfrak{a}}
\newcommand{\bary}{\mathrm{bar}}
\DeclareMathOperator*{\argmin}{arg\,min}

\newcommand{\PC}{\mathcal{P}}

\newcommand{\OC}{\mathcal{O}}

\newcommand{\di}{\text{\rm d}}

\newcommand{\trn}{^{\scriptscriptstyle \top}} 

\newfont{\bfb}{msbm10 scaled 1200}
\newfont{\mfb}{msbm8}

\newcommand{\der}[1]{#1^{\nabla}}

\newcommand{\R}{\mathbb{R}}%

\newcommand{\inner}{\mathrm{int}}%
\newcommand{\res}{\mathrm{res}}%
\newcommand{\rmD}{\text{\bf D}}%
\newcommand{\SC}{\mathcal{S}}%
\newcommand{\N}{\mathbb{N}}

\newcommand{\dif}{\text{\bf D}}

\newcommand{\br}{\mathbb{R}}
\newcommand{\ve}{\varepsilon}

\begin{document}

\title{Remote state estimation problem: towards the data-rate limit along the avenue of the second Lyapunov method}

\author{C.~Kawan\footnote{Institute of Informatics, LMU Munich, Germany (e-mail: christoph.kawan@lmu.de).}
, A.~Matveev\footnote{Department of Mathematics and Mechanics, Saint Petersburg University, St.~Petersburg, Russia (e-mail: almat1712@yahoo.com) and Department of Control Systems and Industrial Robotics, Saint-Petersburg National Research University of Information Technologies Mechanics and Optics (ITMO), Russia.}, A.~Pogromsky\footnote{Department of Mechanical Engineering, Eindhoven University of Technology, Eindhoven, The Netherlands, (e-mail: A.Pogromsky@tue.nl) and Department of Control Systems and Industrial Robotics, Saint-Petersburg National Research University of Information Technologies Mechanics and Optics (ITMO), Russia.}}

\date{}%
\maketitle%

\begin{abstract} 
In the context of control and estimation under information constraints, restoration entropy measures the minimal required data rate above which the state of a system can be estimated so that the estimation quality does not degrade over time and, conversely, can be improved. The remote observer here is assumed to receive its data through a communication channel of  finite bit-rate capacity. In this paper, we provide a new characterization of the restoration entropy which does not require to compute any temporal limit, i.e., an asymptotic quantity. Our new formula is based on the idea of finding a specific Riemannian metric on the state space which makes the metric-dependent upper estimate of the restoration entropy as tight as one wishes.%
\end{abstract}

\section{Introduction}

Recent decades have witnessed a substantially growing attention to networked control systems \cite{HNX07,Maurice_Nathan} and related problems of control and/or state estimation via communication channels with constrained bit-rates; for extended surveys of this area, we refer the reader to \cite{MatveevSavkin_book08,YukBas13,AiT_survey} and references therein. One of the fundamental concerns in this context is to find a minimal data rate between the communication peers above which remote state estimation (\cite{Wong_Brocket}, see also \cite{alfaut} for a related problem) is feasible. In other words, the receiver can reconstruct the current state of the remote system in the real-time regime. The receiver is updated by means of a bit flow with a limited bit-rate. A minimal threshold of this rate, which still ensures that the remote observer is able to keep track of the system, is the quantity one likes to evaluate in a constructive manner. Loosely speaking, the communication rate between the system and the observer has to exceed the rate at which the system ``generates information'', while the latter concept is classically formalized in a form of entropy-like characteristic of the dynamical system at hands. The related mathematical results are usually referred to as Data Rate Theorems (see, e.g. \cite{NFZE07,MatveevSavkin_book08,PartII} and references therein) - their various versions coexist to handle various kinds of observability and models of both the plant and the constrained communication channel.%

Those results deliver a consistent message that the concept of the topological entropy {(TE)} of the system and its recent offshoots provide the figure-of-merit needed to evaluate the channel capacity for control applications; the mentioned modifications of TE are partly aimed to properly respond to miscellaneous phenomena crucial for control problems, like uncertainties in the observed system \cite{SAV06,kawan2017metric,kawan2018optimal}, implications of control actions \cite{CoKaNa13,Nair13,CoKa09,rungger2017invariance}, the decay rate of the estimation error \cite{Liberzon_Mitra_TAC}, or Lipschitz-like relations between the exactness of estimation and the initial state uncertainty \cite{PartII}. Keeping in mind the relevance of communication constraints in modern control engineering, constructive methods to compute or finely estimate those entropy-like characteristics take on crucial not only theoretical but also practical importance. Several steps have been done in this direction in \cite{PoMaNonlin,PartII,hafstein2019numerical}, where corresponding upper estimates were found by following up the ideas of the second Lyapunov method. Moreover, it was shown that for some particular prototypical chaotic systems of low dimensions, these upper estimates are exact in the sense that they coincide with the true value of the estimated quantity.

Whether these inspiring samples of precise calculations are mere incidents, or, conversely, are particular manifestations of a comprehensive capacity inherent in the employed approach? Confidence in the last option would constitute a rationale for undertaking special efforts aimed to fully unleash the potential of this approach via its further elaboration.
\par
The primary goal of the current paper is to answer the posed question; we show that among the above options, the last one is the true one.
This is accomplished via a sort of a converse result, which is similar in spirit to the celebrated converse Lyapunov theorems.
Among various descendants of TE, we pick the so-called {\it restoration entropy} (RE) \cite{PartII} to deal with. In the previous work \cite{PartII}, it has been shown that an upper estimate of RE can be derived in terms of singular values of the derivative of the system flow, calculated with respect to some metric. Any metric involved in such calculations will result in a valid upper estimate. The main contribution of this paper is a result showing that one can find a metric that lets the corresponding upper estimate get arbitrarily close to the true value of RE.
\par
The tractability of the developed approach is confirmed by a closed-form computation of the restoration entropy for the celebrated Lanford system (see, e.g., \cite{belozyorov}). Meanwhile, computation or even fine estimation of TE and the likes has earned the reputation of an extremely complicated matter \cite{Koiran01}.
\par
The paper is organized as follows. Section \ref{sec:comm} offers the remote state estimation problem statement and presents the main assumptions. Section \ref{sec.main} contains the main results, which are illustrated by an example in Section \ref{sec:example}. Section \ref{sec:conclusions} summarizes the main results of the paper and outlines some open questions. The technical proofs of the main results are collected in the appendices.%

The following notations are adopted in this paper: $\log$ -- logarithm base $2$, $\text{\bf D} f$ -- Jacobian matrix of function $f$, $:=$ -- ``is defined to be'', $B_{\delta}(x)$ -- open ball of radius $\delta$ centered at $x$, $C^0(X,Y)$ -- set of all continuous functions from $X$ to $Y$, $P^t$ with $t\in\R$ -- $t$-th power of a symmetric positive definite matrix $P$. Given a matrix $A$, its singular values repeated in accordance to their multiplicities and ordered in the nonincreasing order, are denoted by $\alpha_i(A)$. If the matrix $A$ is parameterized: $A=A(x)$ or $A=A(t,x)$, then for the sake of brevity, the corresponding singular values will be denoted as $\alpha_i(x)$, or $\alpha_i(t,x)$, provided the choice of the matrix $A$ is clear from the context.%

\section{State estimation via limited bit-rate communication and restoration entropy}\label{sec:comm}

The objective of this section is to outline basic points of recent results that motivate the research reported in this paper.
\par
We consider time-invariant dynamical systems of the form
\begin{equation}\label{given_sys}
  \der{x}(t)=\varphi[x(t)], ~~t \in \mathfrak{T}_+,\quad x(0)\in K \subset \br^n
\end{equation}
in the following two cases:
\begin{enumerate}
\item[{\bf c-t)}] Either $\mathfrak{T}_+ = [0,\infty) \subset \br$, and the symbol $\der{x}$ stands for the derivative of the function:
$\der{x}(t) := \dot{x}(t)$;
\item[{\bf d-t)}] Or $\mathfrak{T}_+$ is the set of integers $t \geq 0$, and $\der{x}$ denotes the forward time-shift by one step:
$\der{x}(t) := x(t+1)$.
\end{enumerate}
In \eqref{given_sys}, $\varphi(\cdot)$ is of class $C^1$ and $K$ is a given compact set of initial states that are of our interest; in the case {\bf d-t)}, all considered time variables assume integer values by default.
\par
We deal with a situation where at a remote site $S_{\text{est}}$, direct observation of the time-varying state $x(t)$ is impossible but a reliable estimate $\widehat{x}(t)$ of $x(t)$ is needed at any time $t\in \mathfrak{T}_+$. Meanwhile, data may be communicated to $S_{\text{est}}$ from another site, where $x(t)$ is fully measured at time $t$. The main issue of our interest arises from the fact that the channel of communication between these sites allows to transmit only finitely many bits per unit time. {\it How large their number must be so that a good estimate can be generated?}
\par
We clarify this issue only in general terms by following \cite{MaPo_automatica,PartII} and \cite[Sec.~3.4]{MatveevSavkin_book08}, and refer the reader there for full details. The {\it channel} is a communication facility that transmits finite-bit data packets; transmission consumes time (during which the channel cannot process new packets). The time $t_\ast$ when the transfer of a packet is started, and its bit-size and contents are somewhat manipulable; the channel is used repeatedly, thus transmitting a flow of messages. Despite all variability in the ways of channel usage, the channel imposes an upper bound $b_+(r)$ on the total number of bits that can be transferred within any interval of duration $r$. On the positive side, there is a way to deliver no less that $b_-(r)$ bits. We consider the case where the two associated averaged per-unit-time amounts of bits converge to a common value $c = \lim_{r \to \infty} b_\pm(r)/r$ as the duration $r$ grows without limits; this value $c$ is called the {\it capacity} of the channel.
\par
In the situation at hands, the {\it observer} is composed of a {\it coder} and a {\it decoder}. The coder is located at the measurement site; its function is to generate the departure times $t_\ast$ and the messages to $S_{\text{est}}$ based on the preceding measurements, as well as on the initial estimate $\widehat{x}(0)$ and its accuracy $\delta>0$.
\begin{equation}
\label{exact}
\| x(0) - \widehat{x}(0)\| < \delta, \quad x(0), \widehat{x}(0) \in K.
\end{equation}
The {\it decoder} is built at the site $S_{\text{est}}$ and has access to $\widehat{x}(0)$ and $\delta>0$; its duty is to generate
an estimate $\widehat{x}(t)$ of the current state $x(t)$ at the current time $t$ based on the messages fully received through the channel prior to this time. Both the coder and decoder are aware of $\varphi(\cdot)$ and $K$ from \eqref{given_sys}.
\par
Given the plant \eqref{given_sys}, the possibility of building a reliable observer is dependent on the capacity $c$ of the employed channel \cite{MaPo_automatica,PartII}. Those values of $c$ that make reliable observation possible are in our focus. Their minimum is defined by the system itself, is called the {\it observability rate} $\mathscr{R}(\varphi,K)$, is measured in bits per unit time, and depends on the concept of observation reliability, as is specified in the following.

\begin{definition}[\cite{MaPo_automatica,PartII}]\label{def.exc}
An observer is said
\begin{enumerate}[{\bf i)}]
\item {\em to observe the system} \eqref{given_sys} if for any $\ve>0$ there is
$\delta(\ve)>0$ such that whenever \eqref{exact} holds with $\delta =\delta(\ve)$, the estimation error never exceeds $\ve$:
\begin{equation*}
\label{any.ext}
\|x(t) - \widehat{x}(t)\| \leq \ve\qquad \forall t \in \mathfrak{T}_+  ;
\end{equation*}
\item {\em to regularly observe} the system \eqref{given_sys} if
the anytime error is uniformly proportional to the initial error: there exist $\delta_\ast>0$ and $G>0$ such that
whenever $\delta\leq\delta_\ast$ in \eqref{exact},
    \begin{equation*}
    \label{horosho0}
    \|x(t) - \widehat{x}(t)\| \leq G \delta  \qquad \forall t \in \mathfrak{T}_+  ;
    \end{equation*}
\item {\em to finely observe}  the system \eqref{given_sys} if
in addition to {\rm ii)} the observation error exponentially decays to zero as time progresses: there exist $\delta_\ast>0$, $G>0$, and $\eta >0$ such that whenever $\delta\leq\delta_\ast$ in \eqref{exact}, the following holds:
    \begin{equation*}
    \label{horosho}
    \|x(t) - \widehat{x}(t)\| \leq G \delta \mathrm{e}^{-\eta t}  \qquad \forall t \in \mathfrak{T}_+.
    \end{equation*}
\end{enumerate}
\par
The system \eqref{given_sys} is said to be {\em {\bf i)} observable, {\bf ii)} regularly observable,
or {\bf iii)} finely observable} via a given communication channel if there is an observer that firstly, operates through this channel and secondly, observes, regularly observes,
or finely observes the system, respectively.
\par
The infimum value of the rates $c$ such that \eqref{given_sys} is observable in the selected (from {\rm i)--iii)}) sense via any channel with capacity $c$ is denoted by $\mathscr{R}(f,K)$, where the index $_{\text{\rm o}}, _{\text{\rm reg}}$, and $_{\text{\rm fine}}$ is attached to $\mathscr{R}$ in the cases {\rm i), ii)}, and {\rm iii)}, respectively.
\end{definition}

\par
The rates $\mathscr{R}_{\text{\rm o}}(\varphi,K)$, $\mathscr{R}_{\text{\rm reg}}(\varphi,K)$, $\mathscr{R}_{\text{\rm fine}}(\varphi,K)$ could be viewed as an answer to the key question posed at the start of the section, but only if a method to effectively compute them be available. Their definitions do not directly suggest such a method since they refer to complementing the system with an a priori uncertain object -- an observer. So the first step is to get rid of this uncertainty and to fully express those rates via features of the system only.
\par
Under certain technical assumptions, this feature is the {\it topological entropy} $H_{\text{top}}(\varphi,K)$ of the plant \eqref{given_sys} in the case of  $\mathscr{R}_{\text{\rm o}}(\varphi,K)$, which is simply equal to $H_{\text{top}}(\varphi,K)$; we refer the reader to \cite{Don11} for the definition of $H_{\text{top}}(\varphi,K)$. In the case of both $\mathscr{R}_{\text{reg}}(\varphi,K)$ and $\mathscr{R}_{\text{fine}}(\varphi,K)$, this is another characteristic of the system. It is introduced in \cite{PartII} for the continuous-time case, extended to the discrete-time one in \cite{Kawan_Entropy}, and is called the {\it restoration entropy} $H_{\text{res}}(\varphi,K)$.
\par
To recall its definition, we emphasize once more the following.

\begin{ass}\label{ass.smooth}
The map $\varphi:\R^n \rightarrow \R^n$ is of class $C^1$ and the set $K \subset \R^n$
is compact and forward-invariant.
\end{ass}

Further, $x(t,a)$ denotes the solution of the equation \eqref{given_sys} that starts with $x(0)=a$, and $\varphi^t$ stands for the time-$t$ map of the system, $a \mapsto x(t,a)$.
\par
Let a duration $\tau \in \mathfrak{T}_+$, a state $a\in\R^n$, and an ``error tolerance level'' $\delta >0$ be given. We denote by $N_{\varphi,K}(\tau,a,\delta)$ the minimal number of open $\delta$-balls needed to cover the image $\varphi^\tau[B_\delta(a)\cap K]$. Then
\begin{equation}\label{res_entropy}
H_{\text{res}}(\varphi,K):= {\lim_{\tau \to \infty}} \frac{1}{\tau}\varlimsup_{\delta \to 0} \sup_{a \in K} \log N_{\varphi,K}(\tau,a,\delta).%
\end{equation}
Here $\lim_{\tau \to \infty}$ exists by Fekete's lemma (see \cite[Lem.~2.1]{CoKaNa13} for a proof) and is due to the  subadditivity of the concerned quantity in $\tau$. The interest in \eqref{res_entropy} is caused by the fact that \cite{PartII}%
\begin{equation*}
  H_{\text{res}}(\varphi,K) = \mathscr{R}_{\text{reg}}(\varphi,K) = \mathscr{R}_{\text{fine}}(\varphi,K).
\end{equation*}
\par
Also, $H_{\text{res}}(\varphi,K) \geq H_{\text{top}}(\varphi,K)$ \cite{PartII,Kawan_Entropy}. Meanwhile, the concerned two concepts are not identical: $H_{\text{top}} < H_{\text{res}}$ for the logistic map \cite[Ex.~5.1]{PoMaPSYCO}, whereas \cite{Kawan_Entropy} offers a characterization of the hyperbolic systems with $H_{\text{top}} = H_{\text{res}}$ and suggests that $H_{\text{top}} = H_{\text{res}}$ is a relatively rare occurrence.
\par
The restoration entropy can be also linked to the {\it finite-time Lyapunov exponents} $t^{-1}\ln \alpha_i(t,x)$, where $\alpha_1(t,x) \geq \ldots \geq \alpha_n(t,x)$ are the singular values of the Jacobian matrix $\rmD\varphi^t(x)$. It is convenient for us to divide these exponents by $\ln 2$, which results in the change of the logarithm base:%
\begin{equation*}
  \Uplambda_i(t,x) := \frac{1}{t}\log \alpha_i(t,x).%
\end{equation*}
As is shown in \cite[Thm.~11]{PartII},%
\begin{align}\label{reslyapexp}
  H_{\res}(\varphi,K) &= \max_{x\in K} \varlimsup_{t \rightarrow \infty}\sum_{i=1}^n \max\{0,\Uplambda_i(t,x)\} \\
 &= \lim_{t\rightarrow\infty} \max_{x\in K}\sum_{i=1}^n\max\{0,\Uplambda_i(t,x)\}%
\end{align}
provided that the following holds.

\begin{ass}\label{ass.closeur}
The set $K$ is the closure of its interior.%
\end{ass}

\begin{remark}[\cite{PartII}]
\label{remineq}
If Assumption {\rm \ref{ass.closeur}} is dropped, the first relation from \eqref{reslyapexp} remains true provided that $\leq$ is put in place of $=$ in it.
\end{remark}

\par
The topological entropy and Lyapunov exponents are classic and long-studied concepts. However, their use in assessing $H_{\text{res}}$ is highly impeded by the fact that constructive practical evaluation of both $H_{\text{top}}$ and the limit in \eqref{reslyapexp} has earned the reputation of an intricate matter and is in fact a long-standing challenge that still remains unresolved in many respects.
\par
We offer an alternative machinery for evaluation of $H_{\text{res}}$. By the foregoing, it can also be used to upper estimate $H_{\text{top}}$ and to quantify finite-time Lyapunov exponents for large times.%

\section{Main ideas and results}\label{sec.main}

Due to practical needs, the definition \eqref{res_entropy} operates with balls in the Euclidean metric. Meanwhile, balls and Lyapunov exponents are standard in studies of dynamics on Riemannian manifolds, where its own metric tensor is assigned to every point. A homogeneous, state-independent metric was used in Section \ref{sec:comm}. The idea to modify the material of Section \ref{sec:comm} via assigning its own metric tensor to any point $x \in \br^n$, thus altering the neighboring balls and the finite-time Lyapunov exponents, may look as nothing but a complication with no reason. However, this hint is a keystone for a method that not only aids in practical estimation of $H_{\text{res}}$ but also carries a potential to compute $H_{\text{res}}$ with as high exactness as desired.
\par
To specify this, we denote by $\SC$ the linear space of the real symmetric $n \times n$ matrices, and by $\SC^+ \subset \SC$ the subset of the positive definite ones. A continuous mapping $P:K \rightarrow \SC^+$ gives rise to a Riemannian metric on $K$ \cite{Chavel06} by defining the state-dependent inner product\footnote{Here $P$ is not a Riemannian metric in the strict sense, insofar as $K$ is not necessarily a manifold.}%
\begin{equation*}
  \langle v,w \rangle_{P,x} := \langle P(x)v,w \rangle,%
\end{equation*}
where $\langle \cdot, \cdot \rangle$ is the standard Euclidean inner product.
\par
This Riemannian metric produces its own finite-time Lyapunov exponents and, on a more general level, singular values of $A(x) := \rmD \phi(x), x \in K$ for any map $\phi \in C^1(\br^n \to \br^n)$ such that $\phi(K) \subset K$. Indeed, then $A(x)$ should be viewed as an operator between two spaces, endowed with the inner products $\langle \cdot ,\cdot \rangle_{P,x}$ and $\langle \cdot ,\cdot \rangle_{P,\phi(x)}$, respectively. This obliges to treat the singular values of $A(x)$ as the square roots of the eigenvalues of the operator $\rmD \phi(x)^\ast \rmD \phi(x)$, where $\rmD \phi(x)^\ast$ is the adjoint to $\rmD \phi(x)$ with respect to that pair of inner products.
\begin{lemma}
\label{lem.for.sing}
The singular values  $\alpha_1^P(x|\phi) \geq \cdots \geq \alpha_n^P(x|\phi) \geq 0$ of the matrix $A(x) := \rmD \phi(x)$ in the metric $\langle \cdot,\cdot \rangle_P$ are the square roots of the solutions $\lambda$ for the following algebraic equation:%
\begin{equation}\label{eq_gensveq}
  \det\left[ A(x)\trn P(\phi(x))A(x) - \lambda P(x) \right] = 0.%
\end{equation}
These solutions are also the eigenvalues of the positive semi-definite matrix $B(x)^\top B(x)$, where
\begin{equation}
\label{def.b}
B(x):= P[\phi(x)]^{1/2}A(x)P(x)^{-1/2}.
\end{equation}
\end{lemma}
\begin{proof}
To compute the adjoint to $A(x)$, we observe that
\begin{align*}
  &\langle A(x)v,w \rangle_{P,\phi(x)} = \langle P(\phi(x))A(x)v,w \rangle
  \\
	&= \langle A(x)v,P(\phi(x))w \rangle  = \langle v,A(x)\trn P(\phi(x))w \rangle\\
																	 &= \langle v,P(x)P(x)^{-1}A(x)\trn P(\phi(x))w \rangle \\
																	 &= \langle P(x)v, P(x)^{-1}A(x)\trn P(\phi(x))w \rangle \\
																	 &= \langle v, P(x)^{-1}A(x)\trn P(\phi(x))w \rangle_{P,x}%
\end{align*}
implying $A(x)^* = P(x)^{-1}A(x)\trn P(\phi(x))$. It remains to note that the associated singular value equation%
\begin{equation}\label{eq_gensveq0}
  \det \left[ P(x)^{-1}A(x)\trn P(\phi(x))A(x) - \lambda I_n \right] = 0
\end{equation}
is equivalent to both \eqref{eq_gensveq} and $\det [B(x)^\top B(x) - \lambda I] =0$, which can be seen via multiplying the matrix inside $[\ldots]$ in \eqref{eq_gensveq0} by $P(x) [\ldots]$ and $P(x)^{1/2} [\ldots] P(x)^{-1/2}$ in the first and second case, respectively.
\end{proof}

\subsection{Discrete-time systems}

Now we are in a position to state our first main result.
\begin{theorem}
\label{thm:main_DT}
Let the case {\bf d-t)} and Assumption~{\rm \ref{ass.smooth}} hold. Then the following statements are true (where $\log 0 := - \infty$):
\begin{enumerate}
\item[(i)] Any map $P(\cdot) \in C^0(K,\SC^+)$ gives rise to the following bound on the restoration entropy of the system \eqref{given_sys}:
    \begin{multline}
    \label{res.ineq}
  H_{\res}(\varphi,K) \leq \max_{x\in K} \varSigma^P (x|\varphi),
  \\
  \text{\rm where} \; \varSigma^P (x|\varphi) := \sum_{i=1}^n\max\{0,\log\alpha_i^P(x|\varphi)\}.%
\end{multline}
\item[(ii)] Let the set $K$ satisfy Assumption \ref{ass.closeur} and the Jacobian matrix $\rmD\varphi(x)$ be invertible for every $x \in K$. Then for any $\varepsilon>0$, there exists $P \in C^0(K,\SC^+)$ such that%
\begin{equation*}
  H_{\res}(\varphi,K) \geq \max_{x\in K} \varSigma^P (x|\varphi) - \varepsilon.%
\end{equation*}
\end{enumerate}
\end{theorem}

The proof of this theorem is given in Appendix~\ref{app.dtc}.
\par
In Theorem \ref{thm:main_DT}, $\max_{x \in K}$ is attained since $K$ is compact by Assumption \ref{ass.smooth} and $\alpha_i^P(x|\varphi)$ depend continuously on $x$. The latter holds since the $\alpha_i^P(x|\varphi)$ are the singular values of the matrix \eqref{def.b}, which is continuous in $x$, whereas the singular values continuously depend on the matrix \cite[Cor.~7.4.3.a]{Horn}.
\par
The claim {\bf (i)} converts {\em any} positive definite matrix function $P$ into an upper estimate on $H_{\text{res}}(\varphi,K)$, while taking no limits as $t\rightarrow\infty$, unlike \eqref{res_entropy} or \eqref{reslyapexp}. Meanwhile, {\bf (ii)} shows that this estimate can be made {\em as tight as one wishes} by a proper choice of $P$. So the presented method allows for computation of $H_{\text{res}}(\varphi,K) $ with as high exactness as desired.
\par

The claim (i) of Theorem \ref{thm:main_DT} extends \cite[Thm.~12]{MaPo_automatica} from the case of only constant $P$'s to the class of metrics of the most general type. It is shown in \cite{MaPo_automatica} that the discussed method is largely constructive: intelligent choices of $P$ and simple lower bounds on $H_{\text{res}}$ result in closed-form expressions of $H_{\text{res}}$ in terms of the parameters of some classic prototypical chaotic systems. The claim {\bf (ii)} proves that these are not accidental successes but are manifestations of a fundamental trait of the approach.%

\begin{remark}
\label{re.dd}
\rm
Under the assumptions of (ii) in Theorem \ref{thm:main_DT}, the claim (ii) yields the following new and exact formula:
\begin{equation}
\label{var.formula}
  H_{\res}(\varphi,K) = \inf_{P \in C^0(K,\SC^+)}\max_{x\in K} \varSigma^P (x|\varphi).%
\end{equation}
In fact, $\inf_P$ can be taken only over $P$'s that are extendible to an open neighborhood $V_P$ of $K$ as $C^\infty$-smooth mappings to $\SC^+$. (Since this is not used in the paper, the proof is omitted.)
\end{remark}

\begin{corollary}
\label{cor.dd}
Let $K$ meet Assumption \ref{ass.closeur} and $\mathfrak{G}:= \{ \varphi \in C^1( \br^n \to \br^n): \varphi(K) \subset K , \; \exists \rmD\varphi(x)^{-1} \; \forall x \in K \}$ be endowed with the topology of uniform convergence of the function and its first derivatives on $K$. The map $\varphi \in \mathfrak{G} \mapsto H_{\res}(\varphi,K)$ is upper semi-continuous: $H_{\res}(\varphi,K) \geq \displaystyle{\varlimsup_{\phi \to \varphi, \phi\in \mathfrak{G}}} H_{\res}(\phi,K) \; \forall \varphi \in \mathfrak{G}$.
\end{corollary}

Indeed, the $\alpha^P_i(x|\varphi)$ continuously depend not only on $x$ but also on the map $\varphi \in \mathfrak{G}$, as can be easily seen by the analysis of the arguments in the second paragraph following Theorem \ref{thm:main_DT}. Hence, $\varSigma^P (x|\varphi)$ and $\max_{x\in K}\varSigma^P (x|\varphi)$ also depend on $\varphi \in \mathfrak{G}$ continuously. It remains to note that the infimum (over $P$'s, in our case) of continuous functions is upper semi-continuous; see, e.g., \cite[15.23]{Sch97}.
\par
In a practical setting, Corollary \ref{cor.dd} asserts that any upper estimate of $H_{\text{res}}$ is robust: if an upper estimate is established for a nominal model (let the above $\varphi$ be associated with this model), this estimate remains true, possibly modulo small correction, under small uncertainties in the model and perturbations of its parameters. Here the ``smallness'' of the correction goes to zero as so do the uncertainties and perturbations.%

\subsection{Continuous-time systems}

For continuous-time systems, an analogue to equation \eqref{eq_gensveq} looks as follows:%
\begin{gather}\label{eq_P_CT}
  \det\bigl\{ 2 [P(x)\rmD \varphi(x)]^{\text{sym}} + \dot P(x) -\lambda P(x) \bigr\} = 0,%
\end{gather}
where $B^{\text{sym}} := (B+B^\top)/2$ is the symmetric part of the matrix $B \in \br^{n \times n}$ and $\dot{P}$ is the {\it orbital derivative}:
\begin{equation}
\label{orb.der}
\dot{P}(a) := \lim_{\tau \to 0+} \frac{P[x(\tau,a)] - P[a]}{\tau} .
\end{equation}
This derivative is equal to $\dif P(x) \varphi(x)$ for continuously differentiable $P$'s. However, we shall consider maps $P(\cdot)$ with a limited differentiability, as is described in the following.
\begin{ass}
\label{ass.smoothP}
The map $P(\cdot) \in C^0(K,\SC^+)$ is such that for any $a \in \br^n$, the limit in \eqref{orb.der} exists and is orbitally continuous, i.e., the function $t \mapsto \dot{P}[x(t,a)]$ is continuous.
\end{ass}
\begin{remark}
\label{remroooot}
The $x$-dependent roots of the algebraic equation \eqref{eq_P_CT} are the eigenvalues of the symmetric matrix%
\begin{equation*}
  P(x)^{-1/2} \{2 [P(x)\rmD \varphi(x)]^{\text{\rm sym}} + \dot P(x) \}P(x)^{-1/2}.%
\end{equation*}
\end{remark}
So these roots are real. Being repeated in accordance with their multiplicity, they are denoted by $\varsigma_i^P(x)$, ${i=1,\ldots,n}$.

\begin{theorem}
\label{thm:main_CT}
Suppose that the case {\bf c-t)} and Assumption~{\rm \ref{ass.smooth}} hold. Then the following statements are true:
\begin{enumerate}
\item[(i)] For any map $P(\cdot)$ satisfying Assumption~{\rm \ref{ass.smoothP}},
\begin{equation}
\label{djdj.dff}
  H_{\res}(\varphi,K) \leq \frac{1}{2\ln 2}\max_{x\in K}\sum_{i=1}^n\max\{0,\varsigma_i^P(x)\}.%
\end{equation}
\item[(ii)] Suppose that the set $K$ satisfies Assumption \ref{ass.closeur}. Then for any $\varepsilon>0$, there exists a map $P(\cdot)$ for which Assumption~{\rm \ref{ass.smoothP}} and the following inequality are fulfilled:
\begin{equation*}
  H_{\res}(\varphi,K) \geq \frac{1}{2\ln 2}\max_{x\in K}\sum_{i=1}^n\max\{0,\varsigma_i^P(x)\} - \varepsilon.%
\end{equation*}
\end{enumerate}
\end{theorem}

The proof of this theorem is given in Appendix~\ref{app.ctc}. As in the case of discrete time, analogs of Remark~\ref{re.dd} and Corollary~\ref{cor.dd} follow from Theorem~\ref{thm:main_CT}.%

\section{Example: The Lanford system}\label{sec:example}

Consider the following system:%
\begin{align}\label{lanford}
\begin{split}
  \dot x &= (a-1)x-y+xz \\
  \dot y &= x +(a-1)y+yz,~~~~x,y,z\in\mathbb{R},~~a>0 \\
  \dot z &= az-(x^2+y^2+z^2)
\end{split}
\end{align}
This system is attributed to Lanford and was studied in many publications, see, e.g. \cite{belozyorov}. It is well-known that the system \eqref{lanford} has only two equilibrium points:%
\begin{equation*}
  O_1=[0,0,0]\trn,~~O_2=[0,0,a]\trn.%
\end{equation*}
The value $a=2/3$ is of particular interest, since then there is a heteroclinic orbit connecting the equilibria \cite{belozyorov}.

Let $K$ be some compact forward-invariant set of \eqref{lanford}.

\begin{remark}
In $K$, we necessarily have $z\ge 0$.%
\end{remark}
\par
Indeed, if $z(0)<0$, then  {the third equation of (\ref{lanford}) implies} that the solution escapes to $-\infty$ in finite time ($\dot z<-z^2$).%
\par
In \eqref{eq_P_CT}, the Jacobian matrix is now given as follows:%
\begin{equation*}
  \rmD \varphi(x,y,z)=\left[
\begin{array}{ccc}
a-1+z & -1 & x\\ 1 & a-1+z & y \\ -2x & -2y & a-2z
\end{array}
 \right].%
\end{equation*}
In (i) of Theorem~\ref{thm:main_CT}, we take the following matrix function\footnote{For a more detailed treatment of the metric in this form for related problems of stability of forced oscillations, see \cite{PoMaTAC}.}
\begin{equation}\label{lanford_P}
P(x,y,z)=P_0\mathrm{e}^{w(x,y,z)}=\underbrace{\left[\begin{array}{ccc}1 & 0 & 0\\ 0 & 1 & 0\\ 0 & 0 & 1/2\end{array}\right]}_{=:P_0}\exp\underbrace{\left(\frac{2z}{a}\right)}_{=:w}.%
\end{equation}
Straightforward calculations yield%
\begin{align*}
  P\rmD \varphi{(x)} &= \left[ \begin{array}{ccc} a-1+z & -1 & x\\ 1 & a-1+z & y \\ -x & -y & \frac12(a-2z)\end{array} \right]\mathrm{e}^w,\\
	  \rmD \varphi{(x)}\trn P &= \left[ \begin{array}{ccc} a-1+z & 1 & -x\\ 1 & a-1+z & -y \\ x & y & \frac12(a-2z)\end{array} \right]\mathrm{e}^w,%
\end{align*}
and therefore%
\begin{multline*}
 2 [P(x)\rmD \varphi(x)]^{\text{sym}} = \rmD \varphi(x)\trn P+P\rmD \varphi{(x)}\\
 = \mathrm{e}^w  \left[\begin{array}{ccc}2(a-1+z)& 0& 0\\0 &2(a-1+z)&0\\0 &0&a-2z\end{array}\right].%
\end{multline*}
At the same time,
\begin{align*}
\dot P -\lambda P &= \left(\frac{2\dot z}{a}-\lambda\right)\mathrm{e}^w P_0 \\ &\stackrel{(\ref{lanford})}{=}
\mathrm{e}^w\left(\frac{2}{a}(az-x^2-y^2-z^2)-\lambda\right)\left[\begin{array}{ccc}1 & 0 & 0\\ 0 & 1 & 0\\ 0 & 0 & 1/2\end{array}\right].%
\end{align*}
Finally, the solutions of (\ref{eq_P_CT}) can easily be found:
\begin{eqnarray}
  \lambda_1 &=& 2(a-2z)+2\frac{az-z^2-x^2-y^2}{a}\nonumber\\& &\le -\frac{2z^2}{a}-2z+2a\stackrel{\tiny (z\ge 0)}{\le} 2a,\label{n1lambda}\\
  \lambda_{2,3} &=& 2(a-1+z)+2\frac{az-z^2-x^2-y^2}{a}\nonumber\\ & &\le -\frac{2z^2}{a}+4z+2(a-1)\le 2(2a-1).\label{n2lambda}
\end{eqnarray}
By (i) of Theorem \ref{thm:main_CT}, the following upper estimate holds:%
\begin{align*}
H_\res(K) &\le  \frac{1}{2\ln 2}\max_{K}\big[\max\{0,\lambda_1 \} +2 \max\{ 0, \lambda_{2,3} \} \big]\nonumber\\
&= \frac{1}{2\ln 2}\max\{ \max_K\lambda_1, 2\max_K\lambda_{2,3}, \max_K(\lambda_1+2\lambda_{2,3}) \} .\nonumber
\end{align*}
Maximizing $\lambda_1+2\lambda_{2,3}$ over $z\in\R$ yields%
\begin{align*}
  \max_{x,y,z}(\lambda_1+2\lambda_{2,3}) \le \max_{z}\left[ 6a-4+6z-\frac{6}{a}z^2\right] = \frac{15}{2}a-4.
\end{align*}
By using (\ref{n1lambda}) and (\ref{n2lambda}), we thus arrive at the following.%

\begin{theorem}\label{lanford_upper}
Let $K$ be a compact forward-invariant set of the system (\ref{lanford}) with $a\ge 2/3$. Then%
\begin{equation*}
  H_{\res}(\varphi,K) \le \frac{2(2a-1)}{\ln2}.%
\end{equation*}
\end{theorem}

Our next step is to estimate $H_\res(\varphi,K)$ from below under an extra assumption about $K$. Such an estimate is given by the so-called {\it proximate topological entropy} $H_L(O)$ around the system equilibria $O$,
which is defined as $H_L(O):=\frac{1}{{\rm ln} 2} \sum_{j=1}^n \max \{ {\rm Re} \beta_j; 0\}$ \cite{PartII}. Here $\beta_1, \ldots, \beta_n$ are the eigenvalues of $\rmD \varphi(O)$ repeated in accordance with their algebraic multiplicities and ordered so that their real parts $\text{Re}\, \beta_i$ do not increase as $i$ grows. It is easy to see that now
\begin{equation*}
  H_L(O_1) = \frac{1}{\ln 2}\begin{cases}a & \text{if $0<a\le 1$}\\ 3a-2 & \text{if $a\ge 1$}\end{cases},%
\end{equation*}
\begin{equation*}
  H_L(O_2)=\frac{1}{\ln 2}\begin{cases}0 & \text{if $0<a\le 1/2$}\\ 2(2a-1) & \text{if $a\ge 1/2$}\end{cases},%
\end{equation*}
\begin{equation*}
  \max\{H_L(O_1),H_L(O_2)\}\stackrel{a\ge 2/3}{=}H_L(O_2)=\frac{2(2a-1)}{\ln 2}.%
\end{equation*}
The last relation together with \cite[Cor.~12]{PartII} and Theorem \ref{lanford_upper} gives the following result.%

\begin{theorem}
Assume that $a\ge 2/3$. Let $K$ be any compact forward-invariant set for system (\ref{lanford}), which satisfies Assumption \ref{ass.closeur} and the inclusion $O_2\in{\inner}K$. Then%
\begin{equation*}
  H_{\res}(\varphi,K) = \frac{2(2a-1)}{\ln 2}.%
\end{equation*}
\end{theorem}

At this point, it is worth mentioning that the matrix $P$ from (\ref{lanford_P}) not only provides an upper estimate of the restoration entropy according to the statement (i) of Theorem \ref{thm:main_CT}, but also gives a Riemannian metric for which the lower estimate (see the statement (ii) of Theorem \ref{thm:main_CT}) holds true with $\varepsilon=0$.%

\section{Conclusions, future work, and final remarks}\label{sec:conclusions}

In this paper, we have provided a new characterization of the restoration entropy, which measures the minimal channel capacity above which a system can be regularly or finely observed from a remote location. This characterization is based on the observation that the upper bound provided in \cite{PartII} in terms of the singular values of the system's linearization is true for any choice of the inner product with respect to which singular values are computed, even if the inner product varies from point to point, i.e., constitutes a Riemannian metric. Moreover, the corresponding metric tensor plays a role somewhat similar to the Lyapunov matrix from the direct Lyapunov method. The main results of this paper show that a proper choice of this metric leads to the exact value of the restoration entropy, more precisely, the exact value is attained as the infimum value of the above bound over all metrics. This holds for both discrete- and continuous-time systems.
\par
Moreover, our proofs explicitly disclose a minimizing sequence of metrics $P_N(x)$. This issue is discussed in more details in Appendix~\ref{app.tconcl}. However, constructive methods for building such a sequence or even computing an exact minimizer are beyond the scope of this paper, partly due to the paper length limitations; they form a topic in the schedule of our ongoing and future research.
\par
We further remark that the afore-required compactness for the set $K$ is not, in fact, essential in our results. Indeed, already in the definition of restoration entropy, it is only necessary that the image of a closed $\delta$-ball is compact. Wherever we use compactness in our proof, it can be replaced by some uniformity condition, which is another important fact that distinguishes restoration entropy from topological entropy.
\par
Finally, another possible generalization of our results concerns the abstraction of the state space from a Euclidean space to a differentiable manifold. Using the appropriate technical language, we think that our proofs carry over to this more general setup without the need of introducing new ideas.

\appendix

The remainder of the paper is devoted to the proofs of Theorem \ref{thm:main_DT} and \ref{thm:main_CT}. Section \ref{sec:prelim} summarizes some preliminaries required in the proofs of these results. The main result for the discrete-time case is proven in Section \ref{app.dtc}, while the proof for the continuous-time case is presented in Section \ref{app.ctc}. Section \ref{app.tconcl} highlights a technical conclusion from the proofs: it discusses an explicit formula for a minimizing sequence of metrics.

\section{Technical preamble to proofs of Theorem \ref{thm:main_DT} and \ref{thm:main_CT}}\label{sec:prelim}

From now on, we use the following notations: $\Gl(n,\R)$ -- group of real invertible $n \tm n$ matrices, $I$ -- $n \tm n$ identity matrix, $\|\cdot\|_2$ -- standard Euclidean norm on $\R^n$, $x_i$ -- components of $x \in \br^k$.%

\subsection{Riemannian geometry on the set $\SC^+$}

Since the set $\SC^+$ is open in the vector space $\SC$ of all symmetric $n \times n$-matrices, the first candidate for a Riemannian metric tensor on $\SC^+$ is the Euclidean metric inherited from $\SC$. However, many applications motivate the use of the so-called {\it trace metric} \cite{Bhatia07}. It varies with the point $p \in \SC^+$ and defines the following inner product
on the space $T_p\SC^+$ tangential to $\SC^+$ at $p$, which is in fact a copy of $\SC$:
\begin{equation*}
  \langle v,w \rangle_p := \tr (p^{-1} v p^{-1} w),\quad \forall v,w \in T_p\SC^+ \cong \SC.%
\end{equation*}
This metric makes $\SC^+$ a complete Riemannian manifold \cite{Chavel06}, assigns a special length to any smooth curve, and
gives birth to the Riemannian distance $d(p,q)$, defined as the minimal length of curves bridging the points $p$ and $q$ of $\SC^+$. The minimizer exists, is unique, is called a {\it geodesic}, and is given by \cite[Thm.~6.1.6]{Bhatia07}:
\begin{equation}
\label{def.geod}
 p \#_t\, q := p^{1/2}(p^{-1/2}qp^{-1/2})^t p^{1/2}, \quad t \in [0,1].
\end{equation}
Let $\alpha_1(g) \geq \ldots \geq \alpha_n(g) > 0$ be the singular values of $g \in \Gl(n,\R)$. We put%
\begin{multline}
\label{def.sigma}
  \vec{\sigma}(g) := \big[ \log \alpha_1(g),\ldots,\log \alpha_n(g) \big] \in  \fa^+
  \\
  := \{ \xi \in \br^n : \xi_1 \geq \xi_2 \geq \ldots \geq \xi_n \}%
\end{multline}
and endow $\fa^+$ with the following partial order:%
\begin{equation}
\label{def.prec}
  \xi \preceq \eta \overset{\text{def}}{\Leftrightarrow}
  \begin{cases}
                     \sum_{i=1}^k \xi_i \leq  \sum_{i=1}^k \eta_i & \forall 1 \leq k \leq n-1, \\
		\sum_{i=1}^k \xi_i =  \sum_{i=1}^k \eta_i & \mbox{for } k = n.%
																    \end{cases}%
\end{equation}
Any matrix $g \in \Gl(n,\R)$ defines an action (mapping) $g \ast$ on $\SC^+$, specifically, $p \in \SC^+ \mapsto  g \ast p :=  gpg\trn \in \SC^+$. It is easy to see that $g_1\ast(g_2\ast p) = (g_1g_2)\ast p, I\ast p = p, p_2 = g \ast p_1 \Leftrightarrow p_1 = g^{-1} \ast p_2$, and for any $p,q \in \SC^+$ there exists $g \in \Gl(n,\R)$ (e.g., $g:=q^{1/2}p^{-1/2}$) such that $q=g\ast p$.
\par
The following proposition is based on \cite{Bochi} and lists some properties of the so-called {\it vectorial distance}
\begin{equation}
\label{def.vecd}
  \vec{d}(p,q) := 2 \vec{\sigma}(p^{-1/2}q^{1/2}) \in \fa^+, \quad p,q \in \SC^+.
\end{equation}
\begin{proposition}\label{prop_vectorial_dist}
The following statements hold:
\begin{enumerate}[{\bf a)}]
\item $\vec{d}(I,p) = \vec{\sigma}(p)$ and $\|\vec{d}(p,q)\|_2 = d(p,q) \; \forall p,q \in \SC^+$;
\item $\vec{d}(p_1,q_1) = \vec{d}(p_2,q_2)$ if and only if $g \ast p_1 = p_2$ and $g \ast q_1 = q_2$ for some $g \in \Gl(n,\R)$;
\item $\vec{d}(p,p) = 0$; $\vec{d}(p,q) \preceq \vec{d}(p,r) + \vec{d}(r,q)$;
\item $\vec{d}(q,p) = i(\vec{d}(p,q))$, where
$i(\xi) := -(\xi_n,\xi_{n-1},\ldots,\xi_1)$;
\item The curve \eqref{def.geod} is a {\em geodesic segment} for the vectorial distance, i.e., there is $\xi \in \fa^+$ such that $\vec{d}(p\#_t\, q,p\#_s\, q) = (s-t)\xi$ for all $s \geq t$;
\item $\vec{d}(r \#_{1/2}\, p,r \#_{1/2}\, q) \preceq \frac{1}{2}\vec{d}(p,q)$ for all $p,q,r \in \SC^+$.
\end{enumerate}
\end{proposition}

Here b) shows that the map $g\ast$ is an isometry (preserves distances); c) and d) are akin to the axioms of the metric, and f) replicates formula (4-7)
in \cite{Bochi}.
\begin{proposition}\label{prop_geodesics}
The following holds for geodesics on $\SC^+$:
\begin{gather}
\label{geod.1}
g \ast (p \#_t\, q) = (g \ast p) \#_t\, (g \ast q),
\\
\label{geod.2}
\vec{d}(p \#_t\, q,r \#_t\, o) \preceq (1-t)\vec{d}(p,r) + t\vec{d}(q,o) \; \forall t \in [0,1].
\end{gather}
\end{proposition}
\begin{proof}
Here \eqref{geod.1} holds since the map $g \ast$ is an isometry of the Riemannian space $\SC^+$. To prove \eqref{geod.2}, we
use c) and f) in Proposition \ref{prop_vectorial_dist} and the symmetry $p\#_{\frac{1}{2}} q = q\#_{\frac{1}{2}} p$ \cite[Sec.~4.1, 6.1.7]{Bhatia07} and see that%
\begin{align*}
  \vec{d}(p \#_{1/2}\, q,r \#_{1/2}\, o) &\preceq \vec{d}( p \#_{1/2}\, q, p \#_{1/2}\, o)  + \vec{d}(p \#_{1/2}\, o,r \#_{1/2}\, o) \\
  &= \vec{d}( p \#_{1/2}\, q, p \#_{1/2}\, o)  + \vec{d}(o \#_{1/2}\, p,o \#_{1/2}\, r) \\
	&\preceq \frac{1}{2}\vec{d}(q,o) + \frac{1}{2}\vec{d}(p,r),%
\end{align*}
\eqref{geod.2} with $t = 1/2$. For all numbers $t$ of the form $t = k/2^n$ with integers $k,n$ such that $0 \leq k \leq 2^n$, \eqref{geod.2} follows by induction. Since these numbers are dense in $[0,1]$ and \eqref{def.geod} is continuous in $t$, \eqref{geod.2} extends to the entirety of $[0,1]$ by continuity.
\end{proof}

The point $p \#_{1/2}\, q$ is in fact the barycenter of the two-point set $\{p,q\}$, i.e., the minimizer $r$ of $d(r,p)^2 + d(r,q)^2$. This notion of barycenter has a far-reaching generalization. Specifically, let%
\begin{equation*}
  \Delta_m := \{ \omega \in \R^m : \omega_i \geq 0,\ \sum \omega_i = 1\}%
\end{equation*}
be the standard simplex in $\R^m$, let $\omega \in \Delta_m$, and let $p_1,\ldots,p_m \in \SC^+$. Then the following point is well-defined (see, e.g., \cite{LPa}):%
\begin{equation*}
  \bary(\omega;p_1,\ldots,p_m) := \argmin_{q \in \SC^+} \sum_{i=1}^m \omega_i d(q,p_i)^2
\end{equation*}
and is called a \emph{weighted barycenter} of $p_1,\ldots,p_m$ (also known under other names). The {\it unweighted} barycenter is given by
\begin{equation*}
  \bary(p_1,\ldots,p_m) := \bary( \omega^{=}; p_1,\ldots,p_m), \; \text{where} \; \omega_i^{=} := m^{-1} \; \forall i.
\end{equation*}
There is no closed-form expression for $\bary(\omega;p_1,\ldots,p_m)$, in general. The following theorem from \cite{LPa} (see \cite{Hol} for the unweighted case) characterizes the barycenter as the limit of a certain iterative process, which acts through a binary operation at each step $k=1,2,\ldots$ Specifically, let $i (\!\!\!\!\mod m)$ stand for the remainder after dividing $i$ by $m$. We put $l(k) := \sum_{i=1}^k \omega_{i (\!\!\!\!\mod m)}, s_k := \omega_{k (\!\!\!\!\mod m)}/l(k)$ and consider the following process:%
\begin{equation*}
  \bar{p}_1 := p_1,\qquad \bar{p}_k := \bar{p}_{k-1} \#_{s_k}\, p_{k\ (\!\!\!\!\!\!\mod m)}.
\end{equation*}

\begin{theorem}[\cite{Hol,LPa}]\label{thm_iterative}
$\bary(\omega;p_1,\ldots,p_m) = \lim_{k \rightarrow \infty}\bar{p}_k$.%
\end{theorem}

Now we list some properties of the weighted barycenter.%

\begin{proposition}\label{prop_barycenter}
The following relations hold:%
\begin{gather}
g \ast \bary(\omega;p_1,\ldots,p_m) = \bary(\omega;g\ast p_1,\ldots,g\ast p_m) \qquad \forall g \in \Gl(n,\R);
\label{bar.1}
\\
\label{bar.2}
\vec{d}(u,v) \preceq \omega_m \vec{d}(p_m,p_m') \quad \forall p_1,\ldots, p_m,p_m' \in \SC^+,
\\
\nonumber
\text{\rm where} \quad
\begin{cases}
 u =\bary(\omega;p_1,\ldots,p_{m-1},p_m),\\
  v =\bary(\omega;p_1,\ldots,p_{m-1},p_m');
\end{cases}
\\
\bary(\omega;p_1,\ldots,p_m) = \bary(\omega_{\sigma};p_{\sigma(1)},\ldots,p_{\sigma(m)})
\label{bar.3}
\end{gather}
for any permutation $\sigma$ of $\{1,\ldots,m\}$, where $\omega_{\sigma}$ is the result $(\omega_{\sigma(1)},\ldots,\omega_{\sigma(m)})$ of its action on $\omega$.
\end{proposition}

\begin{proof}
Here \eqref{bar.1} and \eqref{bar.3} are straightforward from the definition of the barycenter and the fact that $g\ast$ is an isometry. To prove \eqref{bar.2}, we consider the sequences $(\bar{p}_k)_{k\in\N}$ and $(\bar{q}_k)_{k\in\N}$ defined as in Theorem \ref{thm_iterative} from the data $(\omega;p_1,\ldots,p_{m-1},p_m)$ and $(\omega;p_1,\ldots,p_{m-1},p_m')$, respectively. By using \eqref{geod.2}, we obtain%
\begin{align*}
 & \vec{d}(\bar{p}_{km},\bar{q}_{km}) = \vec{d}(\bar{p}_{km-1} \#_{s_{km}}\, p_m,\bar{q}_{km-1} \#_{s_{km}}\, p_m') \\
	&\preceq (1 - s_{km}) \vec{d}(\bar{p}_{km-1},\bar{q}_{km-1}) + s_{km} \vec{d}(p_m,p_m') \\
	&= (1 - s_{km}) \vec{d}(\bar{p}_{km-2} \#_{s_{km-1}}\, p_{m-1},\bar{q}_{km-2} \#_{s_{km-1}}\, p_{m-1}) \\  &+ s_{km} \vec{d}(p_m,p_m')
	\preceq (1 - s_{km}) \Bigl[ (1 - s_{km-1}) \vec{d}(\bar{p}_{km-2},\bar{q}_{km-2}) \\  &+ s_{km-1} \underbrace{ \vec{d}(p_{m-1},p_{m-1})}_{=0}\Bigr]
	+ s_{km}\vec{d}(p_m,p_m') \\
	&= (1 - s_{km})(1 - s_{km-1}) \vec{d}(\bar{p}_{km-2},\bar{q}_{km-2}) + s_{km}\vec{d}(p_m,p_m') \\
	&\preceq \ldots \\
  &\preceq \Bigl[\prod_{i=0}^{m-1} (1 - s_{km-i})\Bigr] \vec{d}(\bar{p}_{(k-1)m},\bar{q}_{(k-1)m}) + s_{km}\vec{d}(p_m,p_m').%
\end{align*}
Now we observe that $s_{km} = \omega_m/k$ and%
\begin{align*}
  &\prod_{i=0}^{m-1} (1 - s_{km-i}) = \prod_{i=0}^{m-1} \Bigl(1 - \frac{\omega_{m-i}}{k-1 + \sum_{j=1}^{m-i} \omega_j}\Bigr) \\
	&= \prod_{i=0}^{m-1} \frac{k-1 + \sum_{j=1}^{m-i}\omega_j - \omega_{m-i}}{k-1 + \sum_{j=1}^{m-i} \omega_j} = \frac{k-1}{k}.%
\end{align*}
Hence, we end up with
\begin{align*}
  \vec{d}(\bar{p}_{km},\bar{q}_{km}) \preceq \frac{k-1}{k}\vec{d}(\bar{p}_{(k-1)m},\bar{q}_{(k-1)m}) + \frac{\omega_m}{k}\vec{d}(p_m,p_m').%
\end{align*}
Iterating this estimate all the way down to $k=1$, we find that%
\begin{multline*}
  \vec{d}(\bar{p}_{km},\bar{q}_{km}) \preceq \Bigl(\frac{\omega_m}{k} + \frac{k-1}{k}\frac{\omega_m}{k-1} + \frac{k-1}{k}\frac{k-2}{k-1}\frac{\omega_m}{k-2} + \ldots \\
	 + \prod_{i=1}^k \frac{k-i}{k-i+1}\omega_m\Bigr) \vec{d}(p_m,p_m') = \omega_m \vec{d}(p_m,p_m').%
\end{multline*}
Letting  $k \rightarrow \infty$ with regard to Theorem \ref{thm_iterative} completes the proof.%
\end{proof}

The set of $m$ points $p_i \in \SC^+$ and the tuple of weights $\omega \in \Delta_m$ give birth to a probability measure on $\SC^+$: this is the convex combination $\mu$ of $m$ Dirac measures $\delta_{p_i}$ with the coefficients $\omega_i$.
The weighted barycenter is the minimizer $p$ of the function%
\begin{equation}\label{funct.min}
  p \mapsto \int_{\SC^+} d(p,q)^2 \di \mu(q).%
\end{equation}
This observation permits to extend the notion of the barycenter to any probability measure $\mu$ on the Borel $\sigma$-algebra of $\SC^+$.%

More rigorously, let $\PC^1(\SC^+)$ be the set of all Borel probability measures on $\SC^+$ with finite first moment, i.e.,%
\begin{equation*}
\PC^1(\SC^+) := \left\{ \mu:  \int_{\SC^+} d(I,p)\,  \di \mu(p) < \infty \right\}.%
\end{equation*}
This set can be equipped with the \emph{$1$-Wasserstein metric} \cite{Vil}:
\begin{equation*}
  W_1(\mu,\nu) := \inf_{P \in (\mu|\nu)}\int_{\SC^+ \tm \SC^+} d(p,q)\, \di P(p,q),%
\end{equation*}
where $(\mu|\nu)$ is the set of all probability measures $P$ on $\SC^+ \tm \SC^+$ whose projection along the first and second coordinate coincides with $\mu$ and $\nu$, respectively. Simultaneously \cite[Rem.~6.5]{Vil},
\begin{equation}
\label{weiss.am}
  W_1(\mu,\nu) = \sup_{\|\psi\|_{\mathrm{Lip}} \leq 1}\Bigl[\int_{\SC^+} \psi\, \di \mu - \int_{\SC^+} \psi\, \di \nu\Bigr],
\end{equation}%
where $\|\psi\|_{\mathrm{Lip}}:= \sup_{p_1 \neq p_2 , p_i\in \SC^+} |\psi(p_1) - \psi(p_2)| / d(p_1,p_2) $.
\par
Remark 4.2 and Lemma 5.1 in \cite{Stu} give rise to the following.

\begin{lemma}\label{len.appr}
{\bf (i)} The sums of Dirac measures with identical weights $\frac{1}{m}(\delta_{p_1} + \ldots + \delta_{p_m})$ form, in total, a set that is dense in $\PC^1(\SC^+)$.
{\bf (ii)} For any isometry $\mathscr{I}:\SC^+ \rightarrow \SC^+$, the associated push-forward transformation of probability measures $\mu \mapsto \nu, \nu(E) := \mu(\mathscr{I}^{-1}E)$ isometrically maps $\PC^1(\SC^+)$ into itself.
\end{lemma}

The next proposition asserts existence of the measure barycenter and reports some its properties.%

\begin{proposition}[Lem.~5.1, Thm.~6.3, Prop.~4.3 in \cite{Stu}]\label{prop_barycentermap}
There exists a map $\bary:\PC^1(\SC^+) \rightarrow \SC^+$ with the following properties:%
\begin{enumerate}
\item[(1)] Whenever an isometry $\mathscr{I}:\SC^+ \rightarrow \SC^+$ pushes a probability measure $\mu$ forward into the measure $\nu$,
we have $\bary[\nu] = \mathscr{I}(\bary[\mu])$;%
\item[(2)] The map $\bary[\cdot]$ is $1$-Lipschitz:%
\begin{equation}
\label{lipsch.bary}
  d(\bary[\mu],\bary[\nu]) \leq W_1(\mu,\nu) \quad \forall \mu,\nu \in \PC^1(\SC^+);
\end{equation}
\item[(3)] The minimum of \eqref{funct.min} is attained at $\bary[\mu]$ provided that the right-hand side is well-defined, i.e., $\int_{\SC^+} d(I,q)^2\, \di \mu(q) < \infty$.
\end{enumerate}
\end{proposition}

The last inequality implies that in \eqref{funct.min}, $\int_{\SC^+} d(p,q)^2 \di \mu(q) < \infty$ for any $p \in \SC^+$ since $d(p,q) \leq d(p,I) + d(I,q)$ and the constant $d(p,I)$ is of class $L_2$ with respect to the probability measure.%

\begin{corollary}\label{bar.conmt}
The unweighted barycenter $\bary(p_1,\ldots,p_m)$ continuously depends on $p_i \in \SC^+$.
\end{corollary}

This is immediate from (2) in Proposition \ref{prop_barycentermap}, since \eqref{weiss.am} implies that $W_1(\mu,\nu) \leq \max_i d(p_i, p_i^\prime)$ if $\mu$ and $\nu$ are the sums of equi-weighted Dirac measures at points $p_i$ and $p_1^\prime, \ldots, p_m^\prime$, respectively.%

\subsection{Derivatives of singular values and other matrix functionals}

The {\it $m$-th largest eigenvalue} of a symmetric matrix $A \in \SC$ is the $m$-th term $\lambda_m(A)$ of the nonincreasing sequence $\lambda_1(A) \geq \ldots \geq \lambda_n(A)$ of its eigenvalues, each repeated with regard to its algebraic multiplicity $r_m$. We put $i_m := m-m_{\text{first}}+1$, where $m_{\text{first}}$ is the first position in the sequence that accommodates the number $\lambda_m(A)$, whereas $m_{\text{last}}$ is the last position:%
\begin{align*}
  \lambda_1(A) \geq \ldots &\geq \lambda_{m - i_m}(A) > \lambda_{m_{\text{first}}}(A)  = \ldots = \lambda_{m_{\text{last}}}(A) \\
					 &> \lambda_{m_{\text{last}}+1}(A) \geq \ldots \geq \lambda_n(A).%
\end{align*}
Also let $U$ be an orthogonal matrix that reduces $A$ to a diagonal matrix $U\trn A U$ with nonincreasing diagonal entries. We denote by $\tilde{U}_m$ the matrix that is obtained by depriving $U$ of all columns, except for those in the range from $m_{\text{first}}$ to $m_{\text{last}}$.%

\begin{theorem}[\protect{\cite[Thm.~4.5]{HYe}}]\label{thm_eigs_der}
Let $\OC \subset \R^p$ be an open set and $A:\OC \rightarrow \SC$ be a map of class $C^1$. For any $m =1, \ldots, n, \xi_0 \in \OC$ and unit-length vector $d \in \R^p$, the derivative%
\begin{equation*}
  \lim_{\ve \to 0+} \frac{f(\xi_0+\ve d) - f(\xi_0)}{\ve}%
\end{equation*}
of the function $\xi \in \OC \mapsto f(\xi):= \lambda_m [A(\xi)]$ in direction $d$ exists. It equals the $i_m$-largest eigenvalue of the following $r_m \times r_m$ matrix built from $A(\cdot)$ (along with $\widetilde{U}_m,r_m,i_m,\xi_0$)
\begin{equation*}
  \tilde{U}_m\trn \Bigl[ \sum_{j=1}^p d_j \frac{\partial A}{\partial \xi_j}(\xi_0) \Bigr]\tilde{U}_m.%
\end{equation*}
\end{theorem}

\begin{corollary}\label{cor_singval_der}
For any $C^1$-smooth function $g:\R \rightarrow \Gl(n,\R)$ with $g(0) = I$ and $\dot{g}(0) = H \in \R^{n\tm n}$,
the left-hand side of \eqref{def.sigma} with $g=g(t)$ has the right derivative at $t=0$ and%
\begin{equation*}
  \frac{d}{dt}\vec{\sigma}[g(t)] \Bigl|_{t=0^+}= \frac{1}{2\ln(2)}[\lambda_1(H + H\trn),\ldots,\lambda_n(H + H\trn)].%
\end{equation*}
\end{corollary}

\begin{proof}
The map $t \mapsto A(t) := g(t)g(t)\trn \in \SC$ is of class $C^1$ and
\begin{gather}
\vec{\sigma}[g(t)] \overset{\text{\eqref{def.sigma}}}{=} \frac{1}{2}\left[ \log \lambda_1(A(t)),\ldots,\log \lambda_n(A(t)) \right],
\label{log.diff}
\\
\nonumber
A(0)=I \Rightarrow r_m=n, i_m = m, \widetilde{U}_m = I, \qquad \dot{A}(0) = H + H\trn.
\end{gather}
By applying Theorem \ref{thm_eigs_der} to $p=1, \OC := \br, \xi_0:=0$, we see that the right derivative of $\lambda_m[A(t)]$ exists and is equal to $\lambda_m[\dot{A}(0)] = \lambda_m(H + H\trn)$. This and \eqref{log.diff} complete the proof.
\end{proof}

\begin{lemma}\label{lem.difsr}
The function $(p,q) \in \SC^+ \tm \SC^+ \mapsto \Upsilon (p,q) := p^{-1/2}q^{1/2} \in \Gl(n,\R)$ is of class $C^1$ and its derivative at any point with $p=q$ is given by%
\begin{gather*}
  \rmD \Upsilon(p,p) \left[ \begin{smallmatrix} v_p \\ v_q \end{smallmatrix} \right] = p^{-1/2} h(v_q - v_p) \quad \forall v_p, v_q \in \SC,
\end{gather*}
where $h = h(v) \in \SC$ is the unique solution of the equation%
\begin{equation}\label{eq_matrixroot_der}
  h p^{1/2} + p^{1/2} h =  v \in \SC.%
\end{equation}
\end{lemma}

{\bf Proof:} Existence and uniqueness of the solution for the Lyapunov equation \eqref{eq_matrixroot_der} is a well-known fact; see e.g., \cite[Ch.~3]{Khalil96}. We put $\Xi(p) := p^{-1}, \Omega_\pm(p):=p^{\pm1/2} \; \forall p \in \SC^+$. It is immediate from \cite[Thm.~D2]{Bos07} and \cite[Thm.~1.1]{MoNi17a} that $\Xi, \Omega, \Upsilon$ are $C^1$-smooth and%
\begin{equation*}
  \dif \Xi(p) v = - p^{-1} v p^{-1}, \quad \dif \Omega_+(p) v = h(v) \qquad \forall v \in \SC.%
\end{equation*}
Since $\Omega_- = \Xi \circ \Omega_+$, we have $\dif \Omega_-(p) = \dif \Xi[\Omega_+(p)] \circ \dif \Omega_+(p) \Leftrightarrow
\dif \Omega_-(p) v = - p^{-1/2} h(v)p^{-1/2}$. Finally, $\Upsilon(p,q) = \Omega_-(p) \times \Omega_+(q)$, hence%
\begin{align*}
  \rmD \Upsilon(p,p) \left[ \begin{smallmatrix} v_p \\ v_q \end{smallmatrix} \right] &= \dif \Omega_-(p) v_p \times \Omega_+(p) + \Omega_-(p) \times \dif \Omega_+(p) v_q \\
&= - p^{-1/2} h(v_p)p^{-1/2} \times p^{1/2} + p^{-1/2} \times h(v_q) \\
&= p^{-1/2} [h(v_q) - h(v_p)] = p^{-1/2} h(v_q-v_p). \text{\endproof}
\end{align*}

\subsection{Linear cocycles}\label{sec.cocycl}

A {\it dynamical system} on a metric space $X$ is given by its {\it evolution function (dynamic semiflow)} $\Phi : \mathfrak{T}_+ \times X \to X$, where $\mathfrak{T}_+$ is as in \eqref{given_sys} and $\Phi(0,x) = x \; \forall x, \Phi\big[ s, \Phi(t, x)\big] = \Phi[s+t,x]\; \forall s,t \in \mathfrak{T}_+, x \in X$. For this system, a {\it linear cocycle} is defined as a mapping $(t,x) \mapsto A^{(t)}(x) \in \Gl(k,\R)$ such that%
\begin{equation}\label{def.cycle}
  A^{(0)}(x) = I \mbox{\ and\ } A^{(s+t)}(x) = A^{(s)}[\Phi(t,x)] A^{(t)}(x)
\end{equation}
for all $x \in X,\ s,t \in \mathfrak{T}_+$. In the discrete-time case, the semiflow is determined by $\Psi(\cdot):= \Phi(1,\cdot)$ (specifically, $\Phi(t, \cdot)$ is the $t$-th iterate of $\Psi(\cdot)$ for $t>0$ and the identity map for $t=0$), and the cocycle by its {\it generator} $\mathscr{A}(\cdot):= A^{(1)}(\cdot)$ since \eqref{def.cycle} shapes into
\begin{equation*}
  A^{(0)}(x) = I,\quad A^{(t)}(x) = \mathscr{A}[\Psi^{t-1}(x)]\cdots \mathscr{A}[\Psi^1(x)]\mathscr{A}[x] .
\end{equation*}
Two linear cocycles $A,B:\mathfrak{T}_+ \tm X \rightarrow \Gl(k,\R)$ for a common dynamical system are said to be \emph{conjugate} if there exists a continuous map $V:X \rightarrow \Gl(k,\R)$ (the {\it conjugacy}) such that $B^{(t)}(x) = V[\Phi(t,x)]^{-1}A^{(t)}(x)V(x)$ for all $x\in X$, $t \in \mathfrak{T}_+$. In the discrete-time case, it suffices to test only $t=1$, i.e., the relation $\mathscr{B}(x) = V[\Psi(x)]^{-1}\mathscr{A}(x)V(x) \; \forall x\in X$, where $\mathscr{B}(\cdot):= B^{(1)}(\cdot)$.

\section{Proof of Theorem~\ref{thm:main_DT}}\label{app.dtc}

In this section, the assumptions of this theorem are adopted; in particular, we consider the discrete-time system \eqref{given_sys} (the case {\bf d-t)}) and denote by $\mathbb{Z}_+$ the set of all nonegative integers. The first step towards proving (i) in Theorem \ref{thm:main_DT} is the following.%

\begin{lemma}
There are $C_\pm \in \br$ such that for all $t \in \mathbb{Z}_+$, $x \in K$,
\begin{equation}\label{eq_normeq}
 -\infty < C_- \leq \varSigma^P(t,x) - \varSigma^I(t,x)
\leq C_+ < \infty .
\end{equation}
Here $\varSigma^P(t,x) := \varSigma^P\big[x \big| \varphi^t \big]$ is defined in \eqref{res.ineq} and so $\varSigma^I(t,x)$ is built from $\alpha_i^I(t,x)$, which are the ordinary singular values (by the definition in Lemma~\ref{lem.for.sing}).
\end{lemma}

\begin{proof}
Let $\omega_k(C)$ be the product of $k$ largest singular values of the square matrix $C$ if $k\ge 1$; and $\omega_0(C):=1$. We apply Lem.~\ref{lem.for.sing} to the $t$-th iterate $\phi:=\varphi^t$ of $\varphi$, denote by $A^{(t)}(x)$ and $B^{(t)}(x)$ the respective $A(x)$ and $B(x)$ from that lemma, and see that thanks to the last claim of that lemma,%
\begin{equation*}
\varSigma^P(t,x) - \varSigma^I(t,x) =
  \log \frac{\max_{0 \leq k \leq n}\omega_k(B^{(t)}(x))}{\max_{0 \leq k \leq n}\omega_k(A^{(t)}(x))}.%
\end{equation*}
By Horn's inequality \cite[Ch.~I, Prop.~2.3.1]{BLReitman}, see also \cite[Lem.~8.3]{PoMaNonlin}
\begin{equation*}
  \omega_k(B^{(t)}(x)) \leq \omega_k(P(\varphi^t(x))^{1/2}) \omega_k(A^{(t)}(x)) \omega_k(P(x)^{-1/2}).%
\end{equation*}
Here $x \in K \Rightarrow y:= \varphi^t(x) \in K \; \forall t$ by Assumption \ref{ass.smooth}. Since the singular values continuously depend on the matrix \cite[Thm.~2.6.4]{Horn}, so does $\omega_k$. Hence the following functions are continuous as well:%
\begin{equation*}
  y \in K \mapsto \omega_k(P(y)^{1/2}) \quad \text{and} \quad y \in K \mapsto \omega_k(P(y)^{-1/2})%
\end{equation*}
So the maximum over the compact set $K$ is attained and finite for both quantities. This observation yields the upper estimate in \eqref{eq_normeq}. The lower estimate is obtained likewise by applying Horn's inequality to $\omega_k(A^{(t)}(x))$. Thus we see that \eqref{eq_normeq} does hold.
\end{proof}

\begin{proof}{(of (i) in Theorem \ref{thm:main_DT}):}
By combining \eqref{eq_normeq} with Remark \ref{remineq}, we get%
\begin{equation}\label{first.step}
  H_{\res}(\varphi,K) {\leq} \max_{x\in K} \limsup_{t \rightarrow \infty}\frac{1}{t}\varSigma^P(t,x).%
\end{equation}
The generalized Horn's inequality \cite[Ch.~I, Prop.~7.4.3]{BLReitman}) implies that $\varSigma^P[t+s,x] \leq \varSigma^P[s,\varphi^t(x)] + \varSigma^P[t,x]$. Then, by \cite[Thm.~A.3]{Mor} and the definition of $\varSigma^P(t,x)$, the right-hand side of \eqref{first.step} equals%
\begin{align*}
  \inf_{t>0} \max_{x \in K}\frac{1}{t}\sum_{i=1}^n \max\{0,\log \alpha_i^P(t,x)\} \stackrel{t:=1}{\leq} \max_{x\in K} \sum_{i=1}^n \max\{0,\log \alpha_i^P(x|\varphi)\}.%
\end{align*}
\end{proof}

A key to the proof of {\bf (ii)} in Theorem \ref{thm:main_DT} is the following lemma, which uses the concepts and notations from Subsection \ref{sec.cocycl}.%

\begin{lemma}\label{lem_dc}
Let $A$ be a continuous linear cocycle over a continuous discrete-time dynamical system $\Phi$ on a compact metric space $X$. For any natural $N$, there exists a continuous cocycle $B$, conjugate to $A$, such that for its generator $\mathscr{B}$, the following holds:
\begin{equation}\label{fg.dd}
  \vec{\sigma}[\mathscr{B}(x)] \preceq N^{-1}\vec{\sigma}[A^{(N)}(x)] \quad \mbox{\rm for all\ } x \in X.%
\end{equation}
The involved conjugacy $V$ can be chosen so that $V(x) \in \SC^+$ for all $x$.
\end{lemma}

\begin{proof}
For $x \in X$, let $Q(x)\in \SC^+$ be the unweighted barycenter:%
\begin{equation}
\label{def.qqq}
  Q(x) := \bary(I,A^{(1)}(x)^{-1} \ast I,\ldots,A^{(N-1)}(x)^{-1} \ast I).%
\end{equation}
Recall that we denote the time-$1$ map of $\Phi$ by $\Psi$. By \eqref{bar.1} and \eqref{bar.3}, then the matrix $A^{(1)}(x) \ast Q(x)$ is equal to%
\begin{align*}
  &\bary\bigl(A^{(1)}(x) \ast I,I,A^{(1)}(\Psi(x))^{-1} \ast I,\ldots,A^{(N-2)}(\Psi(x))^{-1} \ast I\bigr) \\
	 &\quad  = \bary\bigl(I,A^{(1)}(\Psi(x))^{-1} \ast I,\ldots,A^{(N-2)}(\Psi(x))^{-1} \ast I,A^{(1)}(x) \ast I\bigr).%
\end{align*}
Meanwhile, $Q[\Psi(x)]$ is the following matrix%
\begin{equation*}
	 \bary\bigl(I,A^{(1)}(\Psi(x))^{-1} \ast I,\ldots, A^{(N-1)}(\Psi(x))^{-1} \ast I\bigr).%
\end{equation*}
For the generator $\mathscr{A}(x) = A^{(1)}(x)$, we thus have%
\begin{multline}\label{eq_preceq}
  \vec{d}\,\big[Q(\Psi(x)),\mathscr{A}(x) \ast Q(x) \big]
 \overset{\text{\eqref{bar.2}}}{\preceq} \frac{\vec{d}\big[A^{(N-1)}(\Psi(x))^{-1} \ast I,\mathscr{A}(x) \ast I \big]}{N} \\
	 \overset{\text{b) in Prop.~\ref{prop_vectorial_dist}}}{=\!=\!=\!=\!=\!=} N^{-1}\vec{d}\big[ I,A^{(N)}(x) \ast I \big].%
\end{multline}
Thanks to Corollary \ref{bar.conmt}, the maps $Q(\cdot)$ and $x \in X \mapsto V(x) := Q(x)^{-1/2} \in \SC^+$ are continuous; also,
$V(x) \ast Q(x) = Q(x)^{-1/2} Q(x) Q(x)^{-1/2} = I$. For the conjugate linear cocycle generated by $\mathscr{B}(x) := V(\Psi(x))\mathscr{A}(x)V(x)^{-1}$, we have
\begin{gather*}
  \vec{\sigma}\big[ \mathscr{B}(x) \ast I \big] = \vec{d}\big[ I,\mathscr{B}(x) \ast I \big]
  \\
  \overset{\text{b) in Prop.~\ref{prop_vectorial_dist}}}{=\!=\!=\!=\!=\!=} \vec{d}\Big\{Q[\Psi(x)]^{1/2} \ast I,Q[\Psi(x)]^{1/2}\mathscr{B}(x) \ast I\Big\} \\
	= \vec{d}\Big\{Q[\Psi(x)],V[\Psi(x)]^{-1}\mathscr{B}(x) \ast I \Big\}
\\
= \vec{d}\Big\{Q[\Psi(x)],\mathscr{A}(x)V(x)^{-1} \ast I\Big\}
	 = \vec{d}\Big\{Q[\Psi(x)],\mathscr{A}(x) \ast Q(x)\Big\}
\\
\stackrel{\eqref{eq_preceq}}{\preceq} N^{-1}\vec{d}\big[ I,A^{(N)}(x) \ast I \big] \overset{\text{\eqref{def.vecd}}}{=} N^{-1}\vec{\sigma}(A^{(N)}(x) \ast I).%
\end{gather*}
To get \eqref{fg.dd}, it remains to note that $\vec{\sigma}(g\ast I) = 2\vec{\sigma}(g)$ by \eqref{def.sigma}.
\end{proof}

\begin{proof}{(of {\bf (ii)} in Theorem \ref{thm:main_DT}):} By \eqref{reslyapexp}, there exists $N$ such that
\begin{align*}
 H_{\res}(\varphi,K) &\geq \frac{1}{N} \max_{x \in K} \sum_{i=1}^n \max \big\{ 0, \log \alpha_i(N,x)\big\} - \ve
 \\
 &\overset{\text{(a)}}{=}  \max_{x \in K} \frac{1}{N} \max_{k=0,\ldots,n}\sum_{i=1}^{k}\log\alpha_i(N,x) - \ve,
 \label{start.iii}
\end{align*}
where $\sum_{i=1}^0 \ldots := 0$, and (a) is due to the fact that, first, we put the singular values in the nonincreasing order, and, second, $\max\{ 0, \log \alpha_i(N,x)\}=0$ whenever $\alpha_i(N,x) \leq 1$. We apply Lemma \ref{lem_dc} to the system \eqref{given_sys} to acquire a cocycle $B$ and conjugacy $V(x) \in \SC^+$. By \eqref{def.sigma}, \eqref{def.prec}, \eqref{fg.dd}, the singular values $\alpha_i[\mathscr{B}(x)]$ of $\mathscr{B}(x)$ are such that%
\begin{gather*}
\sum_{i=0}^{k}\log \alpha_i[\mathscr{B}(x)] \leq \frac{1}{N} \sum_{i=0}^{k}\log \alpha_i(N,x) \quad \forall k=0, \ldots, n
\\
\Rightarrow  H_{\res}(\varphi,K) \geq \max_{x \in K} \max_{k=0,\ldots,n}\sum_{i=1}^{k}\log \alpha_i[\mathscr{B}(x)] - \ve
\\
= \max_{x \in K} \sum_{i=1}^n \max \big\{ 0, \log \alpha_i[\mathscr{B}(x)] \big\} - \ve.
\end{gather*}
It remains to note that $\mathscr{B}(x) = V[\varphi(x)] \mathscr{A}(x)V(x)^{-1}$ for all $x\in K$ by the proof of Lemma \ref{lem_dc}, and so $\alpha_i[\mathscr{B}(x)] = \alpha^P_i(x|\varphi)$ for $P(x) := V(x)^2$ by the last claim of Lemma \ref{lem.for.sing}.
\end{proof}

\section{Proof of Theorem~\ref{thm:main_CT}}\label{app.ctc}

In this section, the assumptions of this theorem are adopted; in particular, the system \eqref{given_sys} is described by ODE $\dot{x} = \varphi(x)$. We use the dynamical flow $\varphi^t$ introduced just after Assumption \ref{ass.smooth} and the associated linear cocycle $(t,x) \mapsto  A^{(t)}(x) := \rmD \varphi^t(x) \in \Gl(n,\R)$.%

{\bf Proof of (i) in Theorem \ref{thm:main_CT}:} is by following the arguments from the proof of Theorem 14 in \cite{PartII}, where $v_d(x) :\equiv 0$. The only difference is that the claim $P(\cdot) \in C^1$ from \cite{PartII} is relaxed to Assumption \ref{ass.smoothP}. But this does not destroy these arguments, which can be seen by inspection that encompasses the proof of the underlying Proposition 8.6 in \cite{PoMaNonlin}.%

In proving {\bf (ii)}, the role of Lemma \ref{lem_dc} is played by the following.%

\begin{proposition}\label{lem_ct}
For every $T>0$, there exists a mapping $P:K \rightarrow \SC^+$ that obeys Asm.~{\rm \ref{ass.smoothP}} and the following inequality
\begin{equation}
\label{prop.mainw}
  \frac{1}{2\ln(2)} \big[ \varsigma_1^P(x),\ldots,\varsigma_n^P(x) \big] \preceq \frac{1}{T}\vec{\sigma}\big[ A^{(T)}(x) \big] \; \forall  x \in K,%
\end{equation}
where $\varsigma_1^P(x) \geq \varsigma_2^P(x) \geq \ldots \geq \varsigma_n^P(x)$ are the solutions of \eqref{eq_P_CT}.
\end{proposition}

Its proof is broken into a string of several lemmata, which deal with a given $T>0$. The notation $\int_a^b \delta_{\gamma(t)}\, \di t$ stands for the measure on $\SC^+$ obtained by pushing forward the Lebesgue measure on $[a,b]$ by a continuous curve $\gamma:[a,b] \rightarrow \SC^+$. We put%
\begin{gather}
\label{def.mmuu}
\mu_{t \to \tau|x} := \frac{1}{\tau-t}\int_t^{\tau} \delta_{A^{(s)}(x)^{-1} \ast I}\, \di s, \quad \forall t<\tau;
\\
\label{def.q}
  Q(x) := \bary[ \mu_{0 \to T|x}] \in \SC^+ , \; x \in K,
  \\
  \text{where} \quad g \ast p = g p g^\top \in \SC^+\; \forall g \in \Gl(n,\R), p \in \SC^+.
  \label{napoom}
\end{gather}

\begin{lemma}\label{lem.qcont}
The mapping \eqref{def.q} is continuous.
\end{lemma}

\begin{proof}
For two continuous $\gamma_i:[a,b] \rightarrow \SC^+$ and $\psi \in C^0(\SC^+,\br)$ with $\|\psi\|_{\mathrm{Lip}} \leq 1$, the following holds for the probability measures $\mu_i = \frac{1}{b-a} \int_a^b \delta_{\gamma_i(t)}\, \di t$ by the definition of ``pushing forward''%
\begin{gather*}
  \left| \int \psi\, \di \mu_1 - \int \psi\, \di \mu_2 \right| = \frac{1}{b-a} \left| \int_a^b \left\{ \psi[\gamma_1(t)] - \psi[\gamma_2(t)]\right\} \; \di t\right|
\\
\leq  \frac{1}{b-a} \int_a^b \left| \psi[\gamma_1(t)] - \psi[\gamma_2(t)]\right| \; \di t \leq \frac{1}{b-a} \int_a^b d[\gamma_1(t), \gamma_2(t)] \; \di t.
\end{gather*}
So \eqref{weiss.am} yields that $W_1(\mu_1,\mu_2) \leq \int_a^b d[\gamma_1(t), \gamma_2(t)] \; \di t \to 0$ as $\gamma_2(t)$ goes to $\gamma_1(t)$ uniformly in $t \in [a,b]$. It remains to note that this uniform convergence does hold for $\gamma_i(t) := A^{(t)}(x_i)^{-1} \ast I$ as $x_2 \to x_1$ and to employ (2) in Proposition \ref{prop_barycentermap}.
\end{proof}

\begin{lemma}\label{lem.push}
For any continuous curve $\gamma:[a,b] \rightarrow \SC^+$ and map $\mathscr{I}:\SC^+ \to \SC^+$, pushing the measure $\mu:= \int_a^b \delta_{\gamma(t)}\, \di t$ forward by $\mathscr{I}$ results in the measure $\nu := \int_a^b \delta_{\mathscr{I}[\gamma(t)]}\, \di t$. In particular, pushing $\mu_{0 \to T|\varphi^t(x)}$ forward by the map $p \in \SC^+ \mapsto A^{(t)}(x)^{-1} \ast p$ results in $\mu_{t\to t+T| x}$.
\end{lemma}

\begin{proof}
For any $\psi \in C^0(\SC^+,\br)$, the change-of-variables formula for the push-forward measure yields%
\begin{equation*}
  \int_{\SC^+} \psi \; \di \nu = \int_{\SC^+} \psi \circ \mathscr{I} \; \di \mu = \int_a^b \psi\{\mathscr{I}[\gamma(t)]\} \; \di t.
\end{equation*}
It remains to note that, by the same argument, the last expression is the integral of $\psi$ with respect to $\int_a^b \delta_{\mathscr{I}[\gamma(t)]}\, \di t$. Now, for $\mu_{0 \to T|\varphi^t(x)}$ and $p \mapsto A^{(t)}(x)^{-1} \ast p$, we obtain the measure $\frac{1}{T}\int_0^T \delta_{A^{(t)}(x)^{-1} \ast [A^{(s)}[\varphi^t(x)]^{-1} \ast I] } \, \di s$ and it remains to invoke \eqref{def.cycle} (with $\Phi(t,x) := \varphi^t(x)$) and \eqref{napoom}.
\end{proof}

Now we recall that a certain $T>0$ was fixed just after Proposition \ref{lem_ct}.%

\begin{lemma}
For any $\ve>0$ and the vector $\mathds{1} \in \br^n$ composed of $1$'s, the following relation holds for all small enough $t>0$:
\begin{equation}
\label{fin.step}
\frac{\vec{d}\big\{ Q[\varphi^t(x)],A^{(t)}(x) \ast Q(x) \big\}}{t} \preceq 2\frac{\vec{\sigma}\big[ A^{(T)}(x) \big]}{T} + \ve \mathds{1}.
\end{equation}
\end{lemma}

\begin{proof}
Since $A^{(s)}(x)^{-1} \ast I \in \SC^+$ is continuous in $s$, there exists $\theta \in (0,T)$ such that whenever $|s-\tau| \leq \theta$ and $\tau \in [0,T]$, we have
\begin{equation}
\label{small.en}
d\big[ A^{(\tau)}(x)^{-1} \ast I ,A^{(s)}(x)^{-1} \ast I \big] \leq T \ep/2 .
\end{equation}
For $t \in (0,\theta)$ and the map \eqref{def.q}, we have
\begin{gather*}
   A^{(t)}(x)^{-1} \ast Q[\varphi^t(x)] \overset{\text{\eqref{def.q}}}{=}
   A^{(t)}(x)^{-1} \ast \bary[\mu_{0\to T|\varphi^t(x)}]
   \\
   \overset{\text{ a,b) in Prop.~\ref{prop_vectorial_dist} and (1) in Prop.~\ref{prop_barycentermap}}}{=\!=\!=\!=\!=\!=\!=\!=\!=\!=\!=\!=\!=\!=\!=}
  \bary\Bigl[\text{the measure addressed in Lem.~\ref{lem.push}}\Bigr]
       \\
\overset{\text{Lem.~\ref{lem.push}}}{=\!=\!=\!=}
\bary[\mu_{t\to t+T| x}];
\end{gather*}
and%
\begin{gather*}
  \vec{d}\left\{Q[\varphi^t(x)],A^{(t)}(x) \ast Q(x) \right\}
  \\
   \overset{\text{b) in Prop.~\ref{prop_vectorial_dist}}}{=\!=\!=\!=} \vec{d}\left\{A^{(t)}(x)^{-1} \ast Q[\varphi^t(x)],Q(x) \right\}
   \\
   = \vec{d}\Bigl( \bary [\mu_{t\to t+T| x}], \bary [\mu_{0\to T| x}]\Bigr).%
\\
\mu_{t \to t+T|x} \overset{\text{\eqref{def.mmuu}}}{=} \underbrace{(1 - t/T)}_{=:a_t} \underbrace{\mu_{t \to T|x}}_{=: \mu} + \underbrace{t/T}_{=: b_t} \mu_{T\to T+t|x},
\\
  \mu_{0\to T| x}
	\overset{\text{\eqref{def.mmuu}}}{=} t/T \mu_{0 \to t|x}+  \Bigl(1 - t/T\Bigr) \mu_{t \to T|x}.%
\\
  \vec{d}\bigl(\bary[\mu_{t \to t+T|x}],\bary[ \mu_{0\to T| x}]\bigr) \\
	\overset{\text{c) in Prop.~\ref{prop_vectorial_dist}}}{\preceq} \vec{d}\bigl(\bary[\mu_{t \to t+T|x}],\bary [a_t\mu + b_t\delta_{A^{(T)}(x)^{-1}\ast I} ]\bigr) \\
	+ \vec{d}\bigl(\bary[a_t\mu + b_t\delta_{A^{(T)}(x)^{-1}\ast I}],\bary[a_t\mu + b_t\delta_I]\bigr) \\
	+ \vec{d}\bigl(\bary[a_t\mu(t) + b_t\delta_I],\bary[\mu_{0\to T| x}]\bigr).%
\end{gather*}
The addends to the right of $\preceq$ are sequentially denoted by $A_1,A_2,A_3$. Thanks to a) in Proposition \ref{prop_vectorial_dist} and \eqref{lipsch.bary}, we have%
\begin{gather*}
  \|A_3\|_2 \leq W_1\left( a_t \mu + b_t\delta_I, a_t \mu + b_t \mu_{0\to t|x} \right) \\
	\overset{\text{\eqref{weiss.am}, \eqref{def.mmuu}}}{=\!=\!=\!=\!=} b_t\sup_{\|\psi\|_{\mathrm{Lip}} \leq 1}\Bigl( \int_{\SC^+} \psi\, \di \delta_I - \frac{1}{t}\int_0^t \psi(A^{(s)}(x)^{-1} \ast I)\, \di s \Bigr) \\
	= \frac{1}{T} \sup_{\|\psi\|_{\mathrm{Lip}} \leq 1}  \int_0^t \Bigl[\psi(I) - \psi(A^{(s)}(x)^{-1} \ast I)\Bigr]\, \di s \\
	\leq \frac{1}{T} \int_0^t d\big[I,A^{(s)}(x)^{-1} \ast I \big]\, \di s \overset{\text{\eqref{small.en}}}{\leq} t\ve/2.
\end{gather*}
By arguing likewise, we establish that $\|A_1\|_2 \leq t\ve/2$. Bringing the pieces together and invoking \eqref{prop_vectorial_dist} results in%
\begin{equation}\label{near-the}
  \frac{1}{t}\vec{d}\left\{Q[\varphi^t(x)],A^{(t)}(x) \ast Q(x) \right\} \preceq \ve \mathds{1} + \frac{1}{t}A_2.
\end{equation}
By (i) in Lemma \ref{len.appr}, $\mu$ is the $W_1$-limit of a sequence whose terms are measures of the form $\mu' = \frac{1}{s}(\delta_{p_1} + \ldots + \delta_{p_s})$, where $s \in \N$ and $p_1,\ldots,p_s\in\SC^+$ are individual for any term. Relation \eqref{bar.2} with $m:=s+1, \omega := (a_t/s,\ldots, a_t/s,\beta_t=t/T) \in \Delta_{m}$ yields%
\begin{align*}
  & \vec{d}\Bigl(\bary [ a_t \mu' + b_t\delta_{A^{(T)}(x)^{-1}\ast I}],\bary [a_t\mu' + b_t\delta_I ]\Bigr) \\
	&= \vec{d}(\bary\big[ \omega; p_1,\ldots,p_m, A^{(T)}(x)^{-1} \ast I),\bary(\omega; p_1,\ldots,p_m, I) \big] \\
	&\preceq \frac{t}{T} \cdot \vec{d}(A^{(T)}(x)^{-1} \ast I,I).%
\end{align*}
Letting $\mu^\prime \to \mu$ results in%
\begin{align}\label{inj}
\begin{split}
  \frac{1}{t}A_2 &\preceq \frac{1}{T}\vec{d}(A^{(T)}(x)^{-1} \ast I,I)%
\\
&\overset{\text{\eqref{napoom}}}{=} \frac{1}{T}\vec{d}\big[ A^{(T)}(x)^{-1} [A^{(T)}(x)^{-1}]^\top,I \big] \\
&\overset{\text{\eqref{def.vecd}}}{=} 2 \frac{1}{T}\vec{\sigma} \Big\{ \big[ A^{(T)}(x)^\top A^{(T)}(x)\big]^{1/2} \Big\} \overset{\text{(a)}}{=} 2\frac{1}{T}\vec{\sigma} \big[ A^{(T)}(x) \big],%
\end{split}
\end{align}
where (a) holds since the involved matrices have the same singular values. The proof is completed by injecting \eqref{inj} into \eqref{near-the}.
\end{proof}

\begin{proof}{(of Proposition~\ref{lem_ct}):} The remaining step of this proof largely comes to computing the limit of left-hand side in \eqref{fin.step} as $t \to 0+$. In this left-hand side, the orbital derivative $\dot{Q}(x) := \frac{d}{dt} Q(\varphi^t(x))\bigl|_{t=0^+}$ exists due to \cite[Lem.~4.5]{BNa}. As for the other argument of $\vec{d}[\cdot,\cdot]$, we have%
\begin{multline*}
  \frac{d}{dt} A^{(t)}(x) \ast Q(x) \bigl|_{t=0^+}
	 = \frac{d}{dt} \Big\{\rmD \varphi^t(x)Q(x)\rmD \varphi^t(x)\trn \Big\} \Bigl|_{t=0^+}
\\
\overset{\frac{d}{dt} \rmD \varphi^t(x)|_{t=0^+} = \dif \varphi(x)\, \text{by \cite[Thm.~3.1, Ch.~V]{Hart82}}}{=\!=\!=\!=\!=\!=\!=\!=\!=\!=\!=\!=\!=\!=\!=\!=\!=\!=\!=\!=} \rmD \varphi(x) Q(x) + Q(x)\rmD \varphi(x)\trn.%
\end{multline*}
Formula \eqref{def.vecd} is our incentive to use Lemma \ref{lem.difsr} and study%
\begin{gather*}
 \frac{d}{dt} \Upsilon \big\{Q[\varphi^t(x)],A^{(t)}(x) \ast Q(x) \big\} \bigl|_{t=0^+}  \\
= \rmD\Upsilon\left[ Q(x),Q(x) \right] \left(\begin{smallmatrix} \dot{Q}(x) \\ \rmD \varphi(x) Q(x) + Q(x) \rmD \varphi(x) \trn \end{smallmatrix}\right)
\overset{\text{Lem.~\ref{lem.difsr}}}{=} Q(x)^{-1/2} h,
\end{gather*}
where $h \in \SC$ is the unique solution of the Lyapunov equation%
\begin{equation*}
  h Q(x)^{1/2} + Q(x)^{1/2} h = \rmD \varphi(x) Q(x) + Q(x) \rmD \varphi(x)\trn - \dot{Q}(x).%
\end{equation*}
In the left-hand side of \eqref{fin.step}, we replace $\vec{d}$ by $2 \vec{\sigma}$ based on \eqref{def.vecd}. Then, by Corollary \ref{cor_singval_der}, the left-hand side converges to the following limit as $t \to 0^+$:%
\begin{equation*}
  \frac{1}{\ln(2)} \big\{\lambda_1[ Q(x)^{-1/2}h + h Q(x)^{-1/2}], \ldots, \lambda_n[ Q(x)^{-1/2}h + h Q(x)^{-1/2}]\big\}.%
\end{equation*}
The function $P(x) := Q(x)^{-1} \in \SC^+$ is continuous by Lemma \ref{lem.qcont}; the foregoing and \cite[Thm.~D2]{Bos07} imply that its orbital derivative $\dot{P}(x)$ exists and $\dot{Q}(x) = - P(x)^{-1} \dot{P}(x) P(x)^{-1}$. Multiplying the Lyapunov equation by $Q(x)^{-1/2}$ on the left and right, we get%
\begin{gather*}
Q(x)^{-1/2}h + h Q(x)^{-1/2} =
\\
P(x)^{1/2} \Big\{ \rmD \varphi(x) P(x)^{-1} + P(x)^{-1}\rmD \varphi(x)\trn
\\
\hspace{0.25\textwidth} + P(x)^{-1}\dot{P}(x) P(x)^{-1} \Big\} P(x)^{1/2}
\\
= P(x)^{-1/2} \Big\{  P(x) \rmD \varphi(x) + \rmD \varphi(x)\trn
P(x)
+ \dot{P}(x)  \Big\} P(x)^{-1/2}.%
\end{gather*}
Due to Remark \ref{remroooot}, this means that the left-hand side of \eqref{fin.step} converges to twice the left-hand side of \eqref{prop.mainw} as $t \to 0^+$. By letting $t \to 0^+$ and then $\ve \to 0+$ in \eqref{fin.step}, we arrive at \eqref{prop.mainw}.
\end{proof}

\begin{proof}{(of (ii) in Theorem \ref{thm:main_CT}):} By \eqref{reslyapexp}, \eqref{start.iii} is still valid with some real $N$, now denoted by $T$. Hence,%
\begin{gather*}
 H_{\res}(\varphi,K) \geq  \max_{x \in K} \frac{1}{T} \max_{k=0,\ldots,n}\sum_{i=1}^{k}\log \alpha_i(T,x) - \ve
 \\
 \overset{\text{\eqref{def.sigma},\eqref{def.prec},\eqref{prop.mainw}}}{\geq} \frac{1}{2\ln 2} \max_{x \in K} \max_{k=0,\ldots,n}\sum_{i=1}^{k}\log\varsigma_i^P(x) - \ve
 \\
 =\frac{1}{2\ln 2}\max_{x\in K}\sum_{i=1}^n\max\{0,\varsigma_i^P(x)\} - \varepsilon.%
\end{gather*}
\end{proof}

\section{Technical conclusions from the proofs}\label{app.tconcl}

The above proofs in fact offer explicit formulas for a sequence of Riemannian metrics $P_N(\cdot)$ that asymptotically provides the exact value of the restoration entropy, i.e., the infimum of the r.h.s. in \eqref{var.formula}. In the discrete-time case, such a formula is implied by \eqref{def.qqq} and the considerations just after \eqref{eq_preceq}, and is as follows:
\begin{equation*}
  P_N(x) = \bary(I,\rmD\varphi(x)^{-1} \ast I,\ldots,\rmD\varphi^{N-1}(x)^{-1} \ast I)^{-1}.%
\end{equation*}
In the continuous-time case, an appropriate family of metrics $P_T(x)$, $T>0$, is described by an analogous but technically more sophisticated formula, which involves the barycenter of the continuum of matrices $\rmD\varphi^t(x)^{-1} \ast I$, $t \in [0,T]$.

Developments of effective numerical procedures for building a minimizing sequence of Riemannian metrics, based on both just sketched and other ideas, is basically a topic of our future research. An initial step towards this end was undertaken in \cite{hafstein2019numerical}.%

\section*{Acknowledgements}

A.~Pogromsky acknowledges his partial support by the UCoCoS project which has received funding from the European Union’s Horizon 2020 research and innovation programme under the Marie Skłodowska-Curie grant agreement No 675080 (Secs.~1,7,8). A.~Matveev acknowledges his support by the Russian Science Foundation grant 19-19-00403. C.~Kawan is supported by the German Research Foundation (DFG) through the grant ZA 873/4-1. A preliminary version of this paper was presented at the 2020 IFAC World Congress \cite{KaMaPo_IFAC_2020}.%

\bibliographystyle{plain}
\bibliography{ifac20}

\end{document}